\documentclass[oneside]{amsart}

\usepackage{graphicx}%
\usepackage{multirow}%
\usepackage{amsmath,amssymb,amsfonts}%
\usepackage{amsthm}%
\usepackage{mathrsfs}%
\usepackage[title]{appendix}%
\usepackage{xcolor}%
\usepackage{textcomp}%
\usepackage{manyfoot}%
\usepackage{booktabs}%
\usepackage{algorithm}%
\usepackage{algorithmicx}%
\usepackage{algpseudocode}%
\usepackage{listings}%
\usepackage{tikz-cd}
\usepackage{float}
\usepackage{hyperref}
\usepackage{enumerate}

\newcommand{\FF}{\mathbb{F}}
\newcommand{\NN}{\mathbb{N}}

\newcommand{\scprod}[1]{\left\langle #1 \right\rangle}

\newcommand{\embrace}[3]{\left#1 #3 \right#2}

\newcommand{\lpar}[1]{\embrace{(}{)}{#1}}

\newcommand{\im}{\textnormal{Im}}

\newcommand{\id}{\textnormal{Id}}
\renewcommand{\ker}{\textnormal{Ker}}

\newcommand{\syma}{\textnormal{Sym}^3(A^2)}
\newcommand{\symf}{\textnormal{Sym}^3(K^2)}

\DeclareMathOperator{\disc}{Disc}
\DeclareMathOperator{\QQ}{\mathbb{Q}}
\DeclareMathOperator{\ZZ}{\mathbb{Z}}

\DeclareMathOperator{\spl}{SL}
\DeclareMathOperator{\gl}{GL}
\DeclareMathOperator{\wc}{WC}
\DeclareMathOperator{\pp}{\mathbb{P}}
\DeclareMathOperator{\gal}{Gal}

\DeclareMathOperator{\sym}{Sym}
\DeclareMathOperator{\tr}{Tr}
\DeclareMathOperator{\tors}{tors}

\newtheorem{X}{X}[section]
\newtheorem{corollary}[X]{Corollary}
\newtheorem{lemma}[X]{Lemma}
\newtheorem{proposition}[X]{Proposition}
\newtheorem{theorem}[X]{Theorem}

\theoremstyle{definition}
\newtheorem{definition}[X]{Definition}
\newtheorem{example}[X]{Example}

\newtheorem{remark}[X]{Remark}

\DeclareFontFamily{U}{wncy}{}
\DeclareFontShape{U}{wncy}{m}{n}{<->wncyr10}{}
\DeclareSymbolFont{mcy}{U}{wncy}{m}{n}
\DeclareMathSymbol{\Sh}{\mathord}{mcy}{"58}

\begin{document}

\title[Monogenity of pure cubic number fields and elliptic curves]{Determining monogenity of pure cubic number fields using elliptic curves}

\author{Jordi Guàrdia}
\email{jordi.guardia-rubies@upc.edu}

\author{Francesc Pedret}
\email{francesc.pedret@upc.edu}

\address{Departament de Matemàtiques, Universitat Politècnica de Catalunya}

\begin{abstract}
    We study monogenity of pure cubic number fields by means of Selmer groups of certain elliptic curves. A cubic number field  with discriminant $D$ determines a unique nontrivial $\mathbb{F}_3$-orbit in the first cohomology group of the elliptic curve $E^D: y^2 = 4x^3 + D$ with respect to a certain 3-isogeny $\phi$.  Orbits corresponding to monogenic fields must lie in the soluble part of the Selmer group $S^{\phi}(E^D/\QQ)$, and this gives a criterion to discard monogenity. From this, we can derive bounds on the number of monogenic cubic fields in terms of the rank of the elliptic curve. We can also determine the monogenity of many concrete pure cubic fields assuming GRH.   
\end{abstract}

\date{\today}

\maketitle

\section{Introduction}

The ring of integers of a number field has played a key role in the development of algebraic number theory since its very beginning. The well-known cubic example presented by Dedekind enlightened  the difficulty of the problem and posed two initial problems: the determination of the ring of integers of a number field and its monogenity. The search for solutions to both problems has generated a number of  important concepts and techniques  in number theory. This paper is focused on the second one: we will show how the   geometric techniques introduced by Alpöge, Bhargava and Shnidman (\cite{Bhargava}) to study densities of  non-monogenic fields can be adapted to prove the non-monogenity of many concrete cubic number fields. 

A number field is called monogenic if its ring of integers $\mathcal{O}_L $ has a {\it power integral basis}, that is,  $\mathcal{O}_L = \ZZ[\alpha]$ for some $\alpha \in L$ . Even if  there are efficient algorithms (\cite{MontesAlgorithm}) to determine $\mathcal{O}_L$,  the monogenity is still an open problem from the computational point of view. Many particular results for families of number fields are known (\cite{WILDANGER2000188}, \cite{GAAL1993563}, \cite{GAAL199690}, \cite{Gaal89}), based on solving Thue equations via Diophantine approximation. In this paper, we study monogenity of cubic number fields by means of  elliptic curves, a newer approach introduced by Alpöge, Bhargava and Shnidman in \cite{Bhargava}.  While they prove density results, we adapt their ideas to study concrete cubic fields, and succeed in proving the non-monogenity of most of the considered cases. We also provide bounds on the number of monogenic pure cubic fields based on the rank of certain elliptic curves naturally attached to the field.

Given a relative extension $L/K$ of number fields,   $L$ is said to be \textit{monogenic} over $K$ if $\mathcal{O}_L = \mathcal{O}_K[\alpha]$ for some $\alpha \in L$. There exists a homogeneous  \textit{index form}  $I_{L/K}$ on $n - 1$ variables such that   $L$ is monogenic over $K$ if, and only if, the diophantine equation $I_{L/K}(X_1,\dots,X_{n-1}) = \pm 1$ has a solution in $\mathcal{O}_K$. In this paper, we define $L$ to be \textit{quasi-monogenic} if $I_{L/K}(X_1,\dots,X_{n-1}) = \pm 1$ has a $K$-rational solution. While quasi-monogenity is weaker than monogenity, it has a good arithmetic-geometric description. For a given $k \in K$, we denote by $E^k$ the elliptic curve defined by $Y^2 = 4X^3 + k$. 
 
\begin{theorem}
    \label{theorem: characterization of quasi-monogenity}
    Let $L/K$ be a cubic relative extension of number fields and let $D$ be the discriminant of $L$ over $K$. Assume $\mathcal{O}_L$ is free over $\mathcal{O}_K$. Then, $L$ is quasi-monogenic over $K$ if, and only if, $L$ is isomorphic to $L_P=K[X]/(X^3 - x_0X^2 + D)$ for some $P =(x_0,y_0)\in E^{-27D}(K)$.
\end{theorem}
This suggests a new strategy to identify all quasi-monogenic cubic fields over $K$ with a given discriminant $D$ satisfying the hypothesis in \ref{theorem: characterization of quasi-monogenity}: it suffices  to compute the quotient of $E^{-27D}(K)$ by $\phi_D(E^D(K))$ for a certain 3-isogeny $\phi_D$. Consequently, the computation of the Mordell-Weil group of $E^{-27D}(K)$ becomes the crucial step.

For cubic pure number fields $L=\QQ(\sqrt[3]{m})$,  the discriminant has the specific form $D=-3n^2$ for $n\in\NN$, and we can show that the corresponding field $L_P$ is again a pure cubic field and $L_P\simeq\QQ\lpar{\sqrt[3]{\frac{y_0 - 9n}{y_0 + 9n}}}$.  Indeed, a cubic number field with discriminant of  the form $D=-3n^2$ is always a pure cubic field. We are able to give an efficient algorithm to compute all the quasi-monogenic pure cubic number fields of a given discriminant, given the Mordell-Weil group of $E^{-27D}$. In practice, most fields are not quasi-monogenic, so this allows us to narrow down the possible monogenic fields substantially. As a byproduct of our methods, we obtain sharp bounds for the number of isomorphism classes of quasi-monogenic pure cubic number fields of a given discriminant in terms of $E^{-27D}(\QQ)/\phi_D(E^D(\QQ))$.

Our approach is the following: Alpöge, Bhargava and Shnidman showed in \cite{Bhargava} how to relate $I_{L/K}$ to a class in the $\phi_D$-Selmer group of $E^{-27D}$. If $L/K$ is monogenic, then said class lies in the soluble part of the $\phi_D$-Selmer group. Our methods involve the explicit computation of these classes in order to determine which cubic relative extensions determine cocycles in the soluble part of the $\phi_D$-Selmer group. 

The structure  of the paper is as  follows. In section \ref{section: preliminaries}, we provide a brief survey on monogenity of free rank 3 algebras over a Dedekind domain $A$ and index forms. 

In section \ref{section: jacobian of the index form}, we review the techniques developed by Alpöge, Bhargava, Shnidman and Elkies in \cite{Bhargava} and \cite{Bhargava_2019}; they  show how to introduce  elliptic curves in the monogenity context. The general idea is to relate the monogenity of a cubic number field of discriminant $D$ to the existence of a certain rational point on the elliptic curve $E^{D}:Y^2 = 4X^3 + D$. This  naturally leads to the notion of quasi-monogenity, introduced in section \ref{section: preliminaries}. 

Some basic bounds on the number of quasi-monogenic cubic  fields of fixed discriminant are given  in section \ref{section: first bounds}. In sections \ref{section: cocycle associated to cubic fields} and \ref{section: from points to orbits}, we develop the machinery necessary to develop, in section    \ref{section: pure cubic number fields},  algorithms to determine the monogenity of pure cubic fields assuming the Generalized Riemann Hypothesis (GRH). Finally, in section \ref{section: computations}, we apply these algorithms to determine the quasi-monogenity of 4555 of the 4619 cubic fields appearing in the LMFDB database (\cite{lmfdb}) whose monogenity was still unknown.

\section{Preliminaries on monogenity of cubic algebras}
\label{section: preliminaries}
We give a brief survey on monogenity and index form. The ideas in this section are well known in the case of number fields, but we expose them in the more general case of free algebras over a Dedekind domain.   A scheme-theoretic version of the topic is introduced in  (\cite[Def. 1.2]{ArpinI}). Throughout this section, $A$ is a Dedekind domain and $B$ is a  free $A$-algebra of finite rank $n$. We recall the definition of the discriminant: 

\begin{definition}{\cite[pp. 31,32]{milneANT}} The \textit{discriminant} of  $\beta_1, \dots, \beta_n\in B$ is defined as
    $$
    \disc(\beta_1,\dots,\beta_n) := \det(\tr_{B/A}(\beta_i\beta_j)),
    $$
    where $\tr_{B/A}(\beta)$ denotes the trace of  the $A$-linear map 
    $x \mapsto \beta x$ from $B$ to $B$. 
\end{definition}
\begin{lemma}{\cite[Lem. 2.23]{milneANT}}
    \label{lemma: disc of elements}
    With the same notation as above, let $\gamma_1,\dots, \gamma_n$ be elements of $B$ such that $\gamma_j = \sum a_{ji}\beta_j$, $a_{ij} \in A$. Then, 
    $$
    \disc(\gamma_1,\dots,\gamma_n) = \det(a_{ij})^2\disc(\beta_1,\dots,\beta_n).
    $$
\end{lemma}
\begin{definition}{\cite[p. 33]{milneANT}}
    Let $B$ be a rank $m$ $A$-algebra. The \textit{discriminant of $B$ over $A$}, denoted by $\disc(B/A)$, is defined to be the class of $\disc(\omega_1,\dots,\omega_n)$ in $A/A^{\times 2}$, where $\{\omega_1,\dots,\omega_n\}$ is any $A$-basis of $B$. Lemma (\ref{lemma: disc of elements}) then implies that $\disc(B/A)$ is well defined and is invariant by isomorphisms of $A$-algebras.
\end{definition}

\begin{definition}
    The \textit{index map} of a rank $n$ free algebra $B$ over $A$ is the map $f_{B/A}:B/A \to \bigwedge^n B$ defined by $f_{B/A}(\overline{\gamma}) = 1\wedge \gamma \wedge \cdots \wedge \gamma^{n - 1}$.
\end{definition}
It is shown easily by induction on $n$ that this map is well defined. Given an $A$-basis $b_1,\dots, b_n$ of $B$, the exterior power $\bigwedge^n B$ can be identified with $A$ itself, and then $f_{B/A}(\overline{\gamma})=\operatorname{Disc}(1,\gamma,\dots,\gamma^{n-1}) b_1\wedge\cdots\wedge b_n$. This leads to the following definition.

\begin{definition}{\cite[Def. 1.2]{ArpinI}}
    A free finite rank $A$-algebra $B$ is \textit{monogenic} if there exists $\gamma\in B$ such that $\operatorname{Disc}(1,\gamma,\dots,\gamma^{n-1})\in A^\ast$. Equivalently, if there is $\gamma\in B$ such that $B=A[\gamma]$.
\end{definition}
 
\begin{remark} 
    For a relative extension  $L/K$ of number fields, $L$ is said to be \textit{monogenic over $K$} if its ring of integers $\mathcal{O}_L$ is monogenic over the ring $\mathcal{O}_K$ of integers of $K$ .
\end{remark}

The monogenity condition of an algebra can be checked through  a polynomial equation. Fixing a basis $\mathcal{B}=\{b_1,\dots, b_n\}$ of $B$ as before, we have $f_{B/A}(x_1 b_1+\dots+x_n b_n)=I_{\mathcal {B}}(x_1,\dots,x_n)b_1\wedge\cdots\wedge b_n$ for some homogeneous form $I_{\mathcal{B}}\in A[X_1,\dots X_n]$, and it is clear that $B$ is monogenic if and only if this form represents a unit over $A$. However, in the concrete cases we are  interested in, it is usual to reduce the form $I_{\mathcal{B}}$ to $n-1$ variables.

\begin{lemma}
    Let $A$ be a Dedekind domain and let $B$ be a free $A$-algebra of rank $n$. Then, $B$ has an $A$-basis of the form $\{1,\omega_2,\dots,\omega_n\}$
\end{lemma}
\begin{proof}
    Since $A$ is noetherian, $B \cong A^n$ is also noetherian. Thus, for any $b \in B$, $A[b]$ is a finitely generated submodule of $B$, which means that $B$ is integral over $A$ (see \cite[Prop. 2.4]{milneANT}). 
    
    Now, $B/A$ is torsion free: otherwise, we would have elements $b \in B\setminus A$ and $r, s \in A\setminus\{0\}$ such that $bs = r$. This means $b$ is in the fraction field of $A$. Combined with the fact that $b$ is integral over $A$ and $A$ is integrally closed, we conclude that $b$ is an element of $A$, contradicting our choice of $b$.

    Since $B/A$ is torsion free and finitely generated, it is projective (see \cite[\S 16.3, Cor. 23]{AAlgebra}), so $B \cong A \oplus B/A$. By the classification theorem of modules over Dedekind domains (see \cite[\S 16.2, Thm. 22]{AAlgebra}), we obtain that $B/A$ is free of rank $n - 1$. Pulling back  a basis $\{1,\overline{\omega}_2,\dots, \overline{\omega}_n\}$ of $A\oplus B/A$, we obtain an $A$-basis of $B$ of the form $\{1, \omega_2,\dots, \omega_n\}$, where $\omega_i$ is a lift of $\overline{\omega}_i$ to $B$ for $i = 2, \dots, n$.
\end{proof}

Since the class of $x_1 + x_2\omega_2 + \cdots + x_n\omega_n$ in $B/A$ does not depend on $x_1$, we conclude that $I_{\mathcal{B}}$ does not depend on $X_1$.

\begin{definition}
    The \textit{index form} of a free rank $n$ algebra $B$ over a Dedekind domain $A$ is the $\gl_{n-1}(A)$-class of the degree $n(n-1)/2$ homogeneous form $I_{\mathcal{B}}\in A[X_2,\dots, X_n]$ for any $A$-basis $\mathcal{B}=\{1,b_2,\dots, b_n\}$ of $B$.
It satisfies
$ 
f_{B/A}(x_1 + x_2b_2 + \cdots + x_n b_n)=I_{\mathcal {B}}(x_2, \dots, x_n) 1\wedge b_2 \wedge\cdots\wedge b_n.
$
\end{definition}

From now on, we shall restrict to $n = 3$. In this particular case, there is a tight relation between the index form and the defining polynomial of the extension of the corresponding fraction fields, which leads to the irreducibility of the index form.

\begin{proposition}
    \label{proposition: irreducibility of index form}
    Let $B$ be an integral free $A$-algebra of rank 3 and let $L$ and $K$ be their fields of fractions, respectively. Assume that the characteristic of $A$ is different from 3. Take an $A$-basis $\mathcal{B} = \{1,\omega_2,\omega_3\}$ of $B$. Then, $I_{B/A}(X,1)$ is a defining polynomial for $L$ over $K$ and, in particular, $I_{\mathcal{B}}$ is irreducible.
\end{proposition}
\begin{proof}
    Since the characteristic of $K$ is different from 3, $L/K$ is a separable extension. Take a basis $\mathcal{B}'$ of $L$ over $K$. Since $\mathcal{B}$ is also an $K$-basis for $L$, we have that $I_{\mathcal{B}}$ and $I_{\mathcal{B}'}$ are $\gl_2(K)$-equivalent via the change of basis matrix. In particular, the dehomogenization of the index form of $B$ over $A$ is also a defining polynomial for $L/K$ and, so, is irreducible.
\end{proof}

This proposition is a kind of reciprocal of the following example, which inspires it: 

\begin{example}
    Let $K$ be any field and let $L$ be a cubic separable extension of $K$ with a monic defining polynomial $f = X^3 + aX^2 + bX + c$. Consider the $K$-basis $\mathcal{B} = \{1,\theta,\theta^2\}$ of $L$, where $\theta$ is a root of $f$ in $L$. Then, one can check that $I_{\mathcal{B}}(X + aY,Y)$ is the homogenization of $f$.
\end{example}

The two  results above show that the (class of the) index form is an invariant of the isomorphism class of the free rank 3 $A$-algebras. 

\begin{definition}{\cite[p.104]{Delone2009TheTO}}
    An $A$-basis of $B$ of the form $\{1, \omega_2, \omega_3\}$ is a \textit{normal basis} if $\omega_2\omega_3$ is an element of $A$.
\end{definition}
Note that we can add elements of $A$ to $\omega_2$ and $\omega_3$ without altering $I_{\mathcal{B}}$. Thus, we can always assume the existence of a normal basis, and from now on $\mathcal{B}$ will be a normal basis. We shall denote by $a, b,c,d,l,m,n$ the elements in $A$ such that 
$$
\begin{cases}
    \omega_2\omega_3 = n,\\
    \omega_2^2 = m - b\omega_2 + a\omega_3,\\
    \omega_3^2 = l - d\omega_2 + c\omega_3.
\end{cases}
$$
With this notation, observe that $f_{B/A}(\omega_2) = a(1\wedge\omega_2\wedge\omega_3)$, so the coefficient of $X^3$ in $I_{\mathcal{B}}$ is $a$. Similarly, we can determine the other coefficients by analysing the images of $\omega_3$ and $\omega_2\pm\omega_3$ by $f_{B/A}$, thus obtaining that
\begin{equation}
    I_{\mathcal{B}}(X,Y) = aX^3 + bX^2Y + cXY^2 + dY^3.
\end{equation}
The coefficients $a,b,c,d$ uniquely determine $B$ up to isomorphism. This gives a correspondence between isomorphism classes of free rank 3 $A$-algebras admitting normal bases and $\gl_2(A)$-orbits on the space $\syma$. This correspondence preserves discriminants. We briefly recall the definition of  the discriminant of a form in $\syma$.

\begin{definition}
    Let $f = aX^3 + bX^2Y + cXY^2 + dY^3 \in \syma$.  The \textit{discriminant of $f$} as 
    $$
    \Delta(f) := b^2c^2 + 18abcd - 4ac^3 - 4db^3 - 27a^2d^2.
    $$
\end{definition}
\begin{remark}
    Note that $\Delta(f)$ is equal to the discriminant of $f(X,1)$ as a polynomial in one variable, which we denote by $\Delta(f(X,1))$. Recall that $f(X,1)$ has a double root if, and only if, $\Delta(f(X,1)) = 0$. Denote by $K$ the fraction field of $A$ and $\overline{K}$ an algebraic closure of $K$. Since $f$ is the homogenization of $f(X,1)$, we obtain $\Delta(f) = 0$ if, and only if, $f(X,Y) = (aX - bY)^sg(X,Y)$, for some $a, b\in \overline{K}$, $s > 1$ and $g \in \sym^{3 - s}(\overline{K}^2)$.
\end{remark}

Following  \cite{Bhargava}, we consider the \textit{twisted action}   
of $\gl_2(A)$ on the space $\sym^r(A^s)$  of homogeneous forms of degree $r$ on $s$ variables,
defined by $\gamma\star f(X) = \det(\gamma)^{-1}f(X\cdot\gamma)$.

A direct computation shows that $\Delta(\gamma\star f) = \det(\gamma)^2 f$ for any $f \in \syma$ and $\gamma \in \gl_2(A)$. 
\begin{definition}
    The \textnormal{discriminant} of a $\gl_2(A)$-orbit on $\syma$ is the class of $\Delta(f)$ in $A/A^{\times 2}$, where $f$ is any form in the orbit.
\end{definition}
Under these circumstances, we have a slight generalization of {\cite[Prop. 4.2]{GanGrossSavin}}.
\begin{proposition}
    \label{proposition: GGS bijection}
    There is a discriminant-preserving bijection 
    \begin{alignat*}{4}
        \left\{\parbox{12em}{\textnormal{Isomorphism classes of rank 3            free } A \textnormal{-algebras}}\right\}&\longleftrightarrow&&\left\{\gl_2(A)-\textnormal{orbits on }\syma\right\}\\
        [B]&\longleftrightarrow&&\gl_2(A)\star I_{\mathcal{B}}(X,Y)
    \end{alignat*}
\end{proposition}
\begin{proof}
    The result follows by replacing $\ZZ$ by $A$ in the proof of \cite[Prop. 4.2]{GanGrossSavin}.
\end{proof}

From section \ref{section: jacobian of the index form} onwards, we will study monogenity by means of elliptic curves. The methods we will use give results about the existence of \textit{rational} solutions to the index form equation. Thus, we propose the following definition.

\begin{definition}
    Let $B$ be a rank $n$ $A$-algebra and $\mathcal{B}$ an $A$-basis of $B$. We define $B$ to be \textit{quasi-monogenic} (over $A$) if $I_{\mathcal{B}}$ represents a unit over the field of fractions of $A$.
\end{definition}

In practice, we will see in section (\ref{section: computations}) that most of the fields we will analyze are not quasi-monogenic and, hence, not monogenic. 

\section{The jacobian of the index form}
\label{section: jacobian of the index form}
In this section, we shall describe the connection between monogenity of free cubic algebras and elliptic curves. The results of this section are developed in \cite{Bhargava} and \cite{Bhargava_2019}.

From now on, we assume that $A$ is a Dedekind domain with characteristic not 2 or 3 and that the discriminant of $B$ over $A$ is not 0 (equivalently, the trace form of $B$ over $A$ is non-degenerate). We denote by $K$ the field of fractions of $A$ and $L$ the field of fractions of $B$, so $L/K$ is a separable cubic field extension. 
\begin{lemma} Let $C_{\mathcal{B}}$ the plane projective curve given by  $I_{\mathcal{B}}(X,Y) = Z^3$ on $\pp^2(K)$, where $\mathcal{B}$ is any $A$-basis of $B$.  Then
\begin{itemize}
    \item[a)]  $C_{\mathcal{B}}$ is smooth. 
\item[b)] For any other $A$-basis $\mathcal{B}'$ of $B$, the curves $C_{\mathcal{B}}$ and $C_{\mathcal{B}'}$ are isomorphic over $K$.
\end{itemize}
\end{lemma}
\begin{proof}
  \begin{itemize}
      \item[a)]
    Let $g(X,Y,Z) = I_{\mathcal{B}}(X,Y) - Z^3$ and suppose there is a singular point $P = (x:y:z)$ on $C_{\mathcal{B}}$. The partial derivatives of $g$ will vanish at $P$, so $z = 0$. 
    
    Suppose $y = 0$. Then we have the equality $I_{\mathcal{B}}(1,0) = 0$. By construction, this implies that $1,\omega_2,\omega_2^2$ are linearly dependent, so $K[\omega_2] \subsetneq L$. Since $3 = [L:K] = [L:K[\omega_2]][K[\omega_2]:K]$, we must have $[K[\omega_2] : K] = 1$, i.e. $\omega_2 \in K$. However, this implies that there exist $a, b \in A$ such that $a = b\omega_2$, contradicting the fact that $\{1,\omega_2,\omega_3\}$ is an $A$-basis of $B$ and, in particular, linearly independent. Thus, assume $y \neq 0$. In this case, $P$ is of the form $(x:1:0)$. The condition that $P$ is a singular point implies 
    $$
    I_{\mathcal{B}}(x,1) = \frac{\partial}{\partial x} I_{\mathcal{B}}(x,1) = 0,
    $$
    i.e. $x$ is a double root of $I(x,1)$. However, 
    $$
    \Delta(I_{\mathcal{B}}(x,1)) = \Delta(I_{\mathcal{B}}) \overset{\ref{proposition: GGS bijection}}{=} \disc(L/K) \neq 0,
    $$ 
    so $I_{\mathcal{B}}(X,1)$ cannot have a double root.
    \item[b)] If $\mathcal{B}'$ is another $A$-basis of $B$, then $I_{\mathcal{B}} = \gamma\star I_{\mathcal{B}'}$ for some $\gamma \in \gl_2(A)$. Thus, the change of variables given in affine coordinates by 
    $$
    (x, y)\mapsto (x,y)\cdot\gamma^{-1}:C_{\mathcal{B}'} \to C_{\mathcal{B}}
    $$ 
    is defined over $K$. Its inverse is the map 
    $$
    (x, y)\mapsto (x,y)\cdot\gamma:C_{\mathcal{B}} \to C_{\mathcal{B}'},
    $$ 
    which is also defined over $K$. \qedhere
\end{itemize}
\end{proof}

\begin{definition}
    Let $E$ be an elliptic curve defined over $K$. A \textit{covering of $E$} is a pair $(C,\pi)$, where $C$ is a smooth projective curve over the separable closure $\overline{K}$ and $\pi : C(\overline{K})\to E(\overline{K})$ is a surjective morphism. If $C$ and $\pi$ are defined over $K$, we say $(C,\pi)$ is \textit{defined over $K$}. A \textit{morphism of coverings} $(C,\pi)$ and $(C',\pi')$ of $E$ is a morphism $f: C \to C'$ over $\overline{K}$ such that $\pi'\circ f = \pi$. Given an isogeny  $\phi : E' \to E$ from another elliptic curve $E'$ to $E$,  a \textit{$\phi$-covering} is a covering $(C,\pi)$ of $E$ which is $\overline{K}$-isomorphic to $(E',\phi)$ as an $E$-covering.
\end{definition}

We shall see that the curve $C_{\mathcal{B}}$ introduced above has the structure of a $\phi$-covering for a certain isogeny $\phi$ to be defined. For any $D \in K\setminus\{0\}$, we shall denote by $E^D$ the elliptic curve $Y^2 = 4X^3 + D$ and, for any $F \in \symf$, $C_K$ denotes the projective curve $F(X,Y) = Z^3$.

Let $F \in \symf$. Let $h_F$ and $g_F$ denote the hessian of $f$ and the jacobian derivative of $f$ and $h$, respectively. Explicitly, denoting the partial derivatives of any form $p \in \symf$ with respect to $x$ and $y$ as $p_x$ and $p_y$, respectively,  
\begin{align*}
    h_f &= \frac{1}{4}(f_{xx}f_{yy} - f_{xy}^2),\\
    g_f &= f_x(h_f)_y - f_y(h_f)_x.
\end{align*}
Both are degree 1 covariants by the twisted action of $\gl_2(K)$ on $\symf$. They are related by the syzygy
$$
g_f^2 + 27\Delta(f)f^2 + 4h_f^3 = 0.
$$
This relation implies that 
$$
(x,y,z)\mapsto \lpar{\frac{-h_f(X,Y)}{Z^2}, \frac{g_f(X,Y)}{Z^3}}
$$ 
defines a rational map $\pi_f : C_{f}\to E^{-27D}$, where $D = \Delta(f)$. Thus, $(C_{f}, \pi_f)$ defines a covering of $E^{-27D}$.

Since $D \neq 0$, every $K$-reducible homogeneous binary cubic form with discriminant $D$ is $\spl_2(K)$-equivalent to $f_0 := X^2Y - (D/4)Y^3$. Indeed, such a form $f_1$ has a factor $aX + bY$ for some $a, b \in K$. A suitable linear change of variables given by a matrix in $\spl_2(K)$ changes this factor into $Y$, so $f_1$ is $\spl_2(K)$-equivalent to $Yf_2(X,Y)$, for some binary quadratic homogeneous form $f_2$. By completing squares, scaling coefficients and imposing that the discriminant of $f_1$ is $D$, we have that $f_1$ is $\spl_2(K)$-equivalent to $f_0$. In particular, any form in $\symf$ is $\spl_2(\overline{K})$-equivalent to $f_0$. This induces a $\overline{K}$-isomorphism of coverings $C_f \to C_{f_0}$, and note that $\phi: (x,y,z)\mapsto (y/2,z,x)$ defines an isomorphism from $E^D$ to $C_{f_0}$. The result of these constructions is that $(C_f, \pi_f)$ is a $\phi_D$-covering, where $\phi_D = \pi_{f_0}\circ \phi : E^D \to E^{-27D}$. Under these circumstances, we can give $C_f$ the structure of a homogeneous space over $E^D$ (see \cite[\S X.3]{Silverman} for a survey on homogeneous spaces); let $\psi:C_f \to E^D$ be a $\overline{K}$-isomorphism of coverings. Then, the map 
$$
(P,Q)\longmapsto \psi^{-1}(P + \psi(Q)):E^D\times C \to C
$$
is independent of the choice of $\psi$ and defines a simply transitive action of $E^D$ on $C$. In particular, $E^D$ is the jacobian of $C_f$ (see \cite[\S X, Thm. 3.8]{Silverman}). Explicitly,
\begin{equation}
    \label{equation: isogeny}
    \phi_D:(x,y) \mapsto \lpar{\frac{x^3 + D}{x^2}, \frac{y(x^3 - 2D)}{x^3}}.
\end{equation}
Note that $\phi_D$ is a 3-isogeny with kernel $\langle(0, \pm \sqrt{D})\rangle$. Its dual isogeny is 
$$
\hat{\phi}_D : (x,y) \mapsto \lpar{\frac{x^3 - 27D}{9x^2}, \frac{y(x^3 + 54D)}{27x^3}},
$$
whose kernel is $\langle (0, \sqrt{27D})\rangle$. Observe that this kernel is the image of $(\zeta\sqrt[3]{D}, \pm \sqrt{-3D})$ by $\phi_D$, where $\zeta$ ranges from the cubic roots of the unit. 

By the twisting principle, the first Galois cohomology group $H^1(K, E^D[\phi_D])$ parametrizes isomorphism classes of $\phi_D$-coverings (see \cite[\S 2]{CREUTZ2012673}). Explicitly, the class corresponding to a $\phi_D$-covering $(C, \pi)$ is $[\sigma \mapsto \psi^{\sigma}\psi^{-1}]$, where $\psi$ is any $\overline{K}$-isomorphism of coverings between $C$ and $E^D$. In the next lemma, $\sim$ denotes the equivalence relation defined as follows: any two $\phi$-coverings $(C,\pi)$ and $(C',\pi')$ are equivalent if $(C,\pi)$ and $(C',\pi')$ are $K$-isomorphic as $\phi$-coverings.

\begin{lemma}
    Let $B$ be a free cubic $A$-algebra. Then, identifying the set of isomorphism classes of $\phi_D$-coverings with $H^1(K,E^D[\phi_D])$, we have
    $$
    \{(C_{\mathcal{B}}, \pi_{I_{\mathcal{B}}}) | \mathcal{B} \textnormal{ is an }A\textnormal{-basis of }B\}/\sim \ = \{\pm\alpha_B\},
    $$
    for some nonzero $\alpha_B \in H^1(K,E^D[\phi_D])$.
\end{lemma}
\begin{proof}
    Fix an $A$-basis $\mathcal{B}_0$ of $B$ and denote by $\alpha_B$ the element in $H^1(K,E^D[\phi_D])$ corresponding to $(C_{\mathcal{B}_0},\pi_{I_{\mathcal{B}_0}})$. By \ref{proposition: GGS bijection}, for any other basis $\mathcal{B}$, we have that $I_{\mathcal{B}} = \gamma\star I_{\mathcal{B}_0}$ for some $\gamma \in \gl_2(A)$. This gives a map $\psi:C_{\mathcal{B}_0} \to C_{\mathcal{B}}$ given in affine coordinates by $(x,y)\mapsto (x,y)\cdot\gamma^{-1}$. Then, one can check that $\pi_{I_{\mathcal{B}}}\circ\psi = \pm \pi_{I_{\mathcal{B}_0}}$, so $(C_{\mathcal{B}},\pi_{I_\mathcal{B}})$ corresponds to $\pm\alpha_B$. In particular, when $I_{\mathcal{B}}(X,Y) = -I_{\mathcal{B}_0}(X,-Y)$, then $(C_{\mathcal{B}},\pi_{I_\mathcal{B}}) = -\alpha_B$, so both $\alpha_B$ and $-\alpha_{B}$ occur for some suitable choices of $A$-bases. Since $I_{\mathcal{B}}$ is irreducible by \ref{proposition: irreducibility of index form}, $\alpha_B$ is not zero.
\end{proof}

This way, $B$ determines a non-trivial $\mathbb{F}_3$-orbit in $H^1(K, E^D[\phi_D])$, where $\FF_3$ acts by repeated addition. 

Next, let $B'$ be another finite free $A$-algebra of rank 3 and discriminant $D$. Let $L'$ be its field of fractions, which is a cubic extension of $K$, and assume $L'\not\cong L$. Then, by (\ref{proposition: irreducibility of index form}), $L \cong K[X]/I_{L}(X,1)$, so $I_{L}$ is reducible over $L$ whereas $I_{L'}$ is not. Thus, $\alpha_L$ is in the kernel of $H^1(K, E^D[\phi_D]) \to H^1(L, E^D[\phi_D])$ but $\alpha_{L'}$ is not. This proves the following result:

\begin{proposition}{\cite[Lemma 15]{Bhargava}}
    If $L$ and $L'$ are non-isomorphic cubic extensions of $K$ of the same discriminant $D$, then $\alpha_L$ and $\alpha_{L'}$ are linearly independent in $H^1(K, E^D[\phi_D])$.
\end{proposition}
Now, $H^1(K,E^D[\phi_D])$ sits in the middle of the Kummer exact sequence 
\begin{equation}
    \label{equation: Kummer}
    E^{-27D}(K) \longrightarrow H^1(K, E^D[\phi_D]) \longrightarrow H^1(K, E^D).
\end{equation}
Recall that the $H^1(K,E^D)$ is the Weil-Ch\^{a}telet group $\wc(E^D/K)$, which parametrises homogeneous spaces up to $K$-isomorphism. In terms of $\phi_D$ coverings and homogeneous spaces, the maps of the Kummer exact sequence \ref{equation: Kummer} are given by $p \mapsto [(E^D, \tau_{P}\circ\phi_D)]$ and $[(C, \pi)] \mapsto [C]$. Thus, $(C, \pi)$ is in the kernel of the corresponding map if, and only if, $[C]$ is the trivial class in the Weil-Ch\^{a}telet group $\wc(E^D/K)$, i.e. if $C$ has a rational point. To study this kernel, we may restrict ourselves to $\phi_D$-coverings and homogeneous spaces having a point everywhere locally, which can be achieved through the $\phi_D$-Selmer and Tate-Shafarevich groups. 
\begin{definition}{\cite[p. 332]{Silverman}}
    Let $\phi: E \to E'$ be an isogeny over $K$. Then, the $\phi$-Selmer group, denoted by $S^{\phi}(E/K)$, is defined by 
    $$
    S^{\phi}(E/K) = \ker\lpar{H^1(K, E[\phi])\to \prod_{v \in M_K} H^1(K_v, E)},
    $$
    where $M_K$ is a complete set of inequivalent absolute values of $K$. The Tate-Shafarevich group, denoted $\Sh(E/K)$, is defined by 
    $$
    \Sh(E/K) = \ker\lpar{H^1(K, E) \to \prod_{v \in M_K}H^1(K_v, E)}
    $$
\end{definition}
In our context, $S^{\phi}(E)$ is the group of $\phi$-coverings having a point everywhere locally and $\Sh(E/K)$ is the group of homogeneous spaces having a point everywhere locally. These groups are related by the short exact sequence
\begin{equation}
    \label{equation: Selmer-TS sequence}
    0 \longrightarrow \frac{E'(K)}{\phi(E(K))} \overset{\delta}{\longrightarrow} S^{\phi}(E/K) \longrightarrow \Sh(E/K)[\phi]\longrightarrow 0.
\end{equation}
The kernel of $H^1(K, E^D[\phi_D]) \to H^1(K, E^D)$ is exactly the kernel of $S^{\phi_D}(E/K)\to\Sh(E/K)$, so we are particularly interested in the soluble part of $S^{\phi_D}(E^D/K)$. This is the image of $E'(K)/\phi(E(K))$ in $S^{\phi_D}(E^D/K)$ by the map $\delta: P \mapsto [Q^{\sigma}-Q]$, where $Q$ is a preimage of $P$ by $\phi$.

Since $\phi_D$ is a 3-isogeny, we have a natural map $S^{\phi_D}(E^D/K)\to S^3(E^D/K)$, through which we can work with multiplication by 3 instead of $\phi_D$. However, this map is not always injective. 

\begin{proposition}{\cite[Lem. 9.1]{pdescent}}
    Let $\phi:E\to E'$ and $\psi: E' \to E$ be isogenies defined over $K$. Then, the following sequence is exact:
    $$
    \begin{tikzcd}[column sep=0.8 em]
        0\ar[r]&\frac{E'(K)[\psi]}{\phi(E(K)[\psi\phi])}\ar[r]&S^{\phi}(E/K)\ar[r]&S^{\psi\phi}(E/K)\ar[r, "\phi"]&S^{\psi}(E'/K)\ar[r]&\frac{\Sh(E'/K)[\psi]}{\phi(\Sh(E/K)[\psi\phi])}\ar[r]&0.
    \end{tikzcd}
    $$
\end{proposition}
\begin{corollary}
    The map $S^{\phi_D}(E^D/K)\to S^3(E^D/K)$ is injective provided $-27D$ is not a square. Otherwise, its kernel is either trivial or has order 3. 
\end{corollary}

\section{First bounds for the number of monogenic algebras with fixed discriminant}
\label{section: first bounds}
With these constructions, we can already get some bounds for the total number of rank 3 quasi-monogenic  $A$-algebras with discriminant $D$ in terms of the Mordell-Weil group of $E^{D}$. 

\begin{proposition}
    Let $N = \# E^D(K)/3E^D(K)$ and $c = \#\ker(S^{\phi_D}(E^D/K) \to S^3(E^D/K))$. The total number of quasimonogenic cubic $A$-algebras is $\frac{cN - 1}{2}$.
\end{proposition}
\begin{proof}
    The size of $E^{-27D}(K)/\phi_D(E^D(K)) \setminus\{\mathcal{O}\}$ is $cN - 1$. The result follows from the fact that each pair of non-trivial points in $E^{-27D}(K)/\phi_D(E^D(K))$ determines a unique quasi-monogenic cubic $A$-algebra.
\end{proof}
In the particular case when $A = \ZZ$, define 
$$
\delta_D :=
\begin{cases}
    1&\textnormal{if }E^D(\QQ)_{\tors}\simeq \ZZ/3\ZZ, \\
    0&\textnormal{otherwise.}
\end{cases}
$$
Then, by \cite[Thm. 5.3]{knapp}, $\# E^D(K)/3E^D(K) = 3^{r + \delta_D}$. Applying the previous proposition, we obtain the following corollary.
\begin{corollary}
    Let $D$ be an integer such that $-27D$ is not a square in $\ZZ$. Then, the total number of quasi-monogenic cubic number fields of dicriminant $D$ is bounded by $\frac{3^{r + \delta_D} - 1}{2}$. 
\end{corollary}
In the case of pure cubic fields, one has that $-27D$ is a square in $\ZZ$. For example, Dedekind type I fields (to be defined below) have discriminant $-27n^2$, where $n$ is square-free. For such discriminants, one has that the torsion of $E^D$ is trivial (direct consequence of \cite[Thm. 5.3]{knapp}). Thus, we get the following bound. 
\begin{corollary}
    Suppose $D = -27n^2$ with $n$ square-free. Then, the total number of quasi-monogenic cubic number fields is bounded by $\frac{3^{r + 1}-1}{2}$.
\end{corollary}

In section (\ref{section: pure cubic number fields}), we shall improve this bound to $2^r$.

\section{The cocycle associated to cubic fields}
\label{section: cocycle associated to cubic fields}
Now, we shall focus on relative extensions of number fields, i.e. $A = \mathcal{O}_K$ and $B = \mathcal{O}_L$, where $K$ is a number field and $L$ is a cubic extension of $K$. Assume that $\mathcal{O}_L$ is free over $\mathcal{O}_K$. Let $D$ be the discriminant of $L$ over $K$ and let $\beta$ be a primitive element of $L$ over $K$.

\begin{lemma}
    \label{lemma: Third point}
    Suppose we have a cocycle $\zeta: \gal(\overline{K}/K) \to E^{D}[\phi_D]$ given by 
    $$
    \zeta:\sigma\mapsto 
    \begin{cases}
        \mathcal{O} &\textnormal{if }\sigma(\beta) = \beta,\\
        (0:\sqrt{D}:1) &\textnormal{if }\sigma(\beta) = \beta',\\
        Q&\textnormal{if }\sigma(\beta) = \beta'',
    \end{cases}
    $$
    where $\beta'$ and $\beta''$ are different conjugates of $\beta$. Then, $Q = -(0:\sqrt{D}:1)$.
\end{lemma}
\begin{proof}
    Denote $P_D = (0:\sqrt{D}:1)$. We can restrict to $K$-automorphisms of the Galois closure of $L$. We distinguish two cases.  If $D$ is a square, then $L$ is Galois over $K$ and $\gal(L/K) \simeq C_3$. Let $\sigma$ be a generator of $\gal(L/K)$ such that $\sigma(\beta) = \beta'$. Then, $\zeta(\sigma^2) = \zeta(\sigma)^{\sigma} + \zeta(\sigma) = -P$. Otherwise, if $D$ is not a square, then $L$ is not Galois over $K$. Let $\Tilde{L}$ be its Galois closure, equal to $K(\beta, \sqrt{D})$. We have $\gal(\Tilde{L}/K) \simeq S_3$. The Galois group is generated by
    \begin{equation}
        \label{equation: Galois generators}
        \sigma:\binom{\beta\mapsto\beta'}{\sqrt{D}\mapsto \sqrt{D}}, \qquad
        \tau: \binom{\beta\mapsto\beta}{\sqrt{D}\mapsto -\sqrt{D}}
    \end{equation}
    Using the cocycle identity as before, one checks that $\zeta(\sigma^2) = \zeta(\sigma\tau) = -P_D$ and $\zeta(\sigma^2\tau) = P_D$, as we wanted. 
\end{proof}
\begin{lemma}
    \label{lemma: cocycle cubic field}
    The $\mathbb{F}_3$-orbit of the class of
    $$
    \xi_L:\sigma \mapsto 
    \begin{cases}
        \mathcal{O}&\textnormal{if }\sigma(\beta) = \beta\\
        (0:\sqrt{D}:1)&\textnormal{if }\sigma(\beta) = \beta'\\
        -(0:\sqrt{D}:1)&\textnormal{if }\sigma(\beta) = \beta''\\
    \end{cases}
    $$
    is the orbit of $\alpha_L$ in $H^1(K,E^D[\phi_D])$. In particular, if $\beta = \sqrt[3]{n}$ for some $n \in \mathcal{O}_K$, $\xi_L$ admits the expression
    $$
    \xi_L: \sigma \mapsto \log_{\mu_3}\lpar{\frac{\sigma(\beta)}{\beta}}(0:\sqrt{D}:1),
    $$
    where $\mu_3$ is a primitive cubic root of the unit.
\end{lemma}
\begin{proof}
    Let $I_{L/K}$ be the index form of $L$ over $K$. Since $I_{L/K}$ is reducible over $L$, we have that $I_{L/K}$ is $\spl_2(L)$-equivalent to $f_0$. This induces a $L$-isomorphism of coverings $\psi: C_{L/K}\to E^D$. The orbit associated to $L$ is given by $\scprod{\alpha_L}$, where $\alpha_L$ is the class of $\sigma \mapsto \psi^{\sigma}\psi^{-1}$. Since $\psi$ is defined over $L$, any two elements of $G_{K}$ yielding the same restriction to $L$ have the same image by $\alpha_L$. Thus, $\alpha_L$ is of the form 
    $$
    \sigma\mapsto
    \begin{cases}
        \mathcal{O}&\textnormal{if }\sigma(\beta) = \beta,\\
        P&\textnormal{if }\sigma(\beta) = \beta',\\
        Q&\textnormal{if }\sigma(\beta) = \beta''.
    \end{cases}
    $$
    Since $\alpha_L$ is non-trivial, we have $P \neq \mathcal{O}$, so the result follows by lemma (\ref{lemma: Third point}).
\end{proof}
\section{From points of the jacobian to elements in cohomology}
\label{section: from points to orbits}
Our general approach is based on the following idea. The soluble part of $S^{\phi_D}(E^D/\QQ)$ is parametrised by $E^{-27D}(K)/\phi_D(E^D(K))$. Given generators of the latter, if we are able to explicitly describe their image in $S^{\phi_D}(E^D/\QQ)$ and distinguish the classes corresponding to fields, then the problem of determining all the quasi-monogenic cubic number fields is reduced to computing generators of $E^{-27D}(K)/\phi_D(E^D(K))$. The following theorem is the fundamental step towards relating all of these objects. Recall that, for $P$ in $E^D(K)$, we define $L_P = K[X]/(X^3 -x_0X^2 + D).$ 
\begin{lemma}
    \label{lemma: explicit cocycle}
    Let $P = (x_0,y_0) \in E^{-27D}(K) \setminus \phi_D(E^D(K))$. The image of $[P]$ in $S^{\phi_D}(E^D/K)$ is, up to sign, the class of
    $$
    \xi_P:\sigma \mapsto
    \begin{cases}
        \mathcal{O}&\textnormal{if }\sigma(\alpha) = \alpha,\\
        (0:\sqrt{D}:1)&\textnormal{if }\sigma(\alpha) = \alpha',\\
        -(0:\sqrt{D}:1)&\textnormal{if }\sigma(\alpha) = \alpha'',\\
    \end{cases}
    $$
    where $\alpha$ is any primitive element of $L_P$ and $\alpha', \alpha''$ are conjugates of $\alpha$. In particular, if the discriminant of $K(\alpha)$ is $D$, then $K(\alpha)$ is quasi-monogenic.
\end{lemma}
\begin{proof}
    Let $Q$ be a preimage of $P$ by $\phi_D$ and let $\alpha$ be its first component. The relation $\phi_D(Q) = P$ shows that $\alpha^3 - x_0\alpha^2 + D = 0$. Now, formula (\ref{equation: isogeny}) shows that the second component of $Q$ is in $K(\alpha)$. Different elements of $G_{K}$ yielding the same restriction to $K(\alpha)$ map to the same element by $\sigma \mapsto [Q^{\sigma} - Q]$, so the expression for the cocycle follows from lemma (\ref{lemma: Third point}). When $K(\alpha)$ has discriminant $D$, this cocycle is, up to sign, $\xi_{K(\alpha)}$ (see \ref{lemma: cocycle cubic field}), so $K(\alpha)$ defines a soluble orbit of $S^{\phi_D}(E^D/K)$. Thus, $K(\alpha)$ is quasi-monogenic.
\end{proof}
\textit{Proof of theorem \ref{theorem: characterization of quasi-monogenity}}. If $L$ is isomorphic to $L_P$ for some $P \in E^{-27D}(K)$, then $L$ is quasi-monogenic by lemma \ref{lemma: explicit cocycle}. Conversely, suppose $L$ is quasi-monogenic over $K$. Then, $\alpha_L$ is in the soluble part of $S^{\phi_D}(E^D/K)$, i.e. $\alpha_L = [\xi_P]$ for some $P \in E^{-27D}(K)$. This means that $\xi_L - \xi_P$ is a coboundary. By choosing a suitable $\alpha$ in lemma \ref{lemma: explicit cocycle}, we can suppose that $\xi_L = \xi_P$. Then, by the expressions shown in lemmas \ref{lemma: explicit cocycle} and \ref{lemma: cocycle cubic field}, we have that $\gal(\overline{K}/L) = \gal(\overline{K}/L_P)$, so $L = L_P$. \qed
\section{Pure cubic number fields}
\label{section: pure cubic number fields}
We shall focus on pure cubic number fields, i.e. cubic extensions of the form $L = \QQ(\sqrt[3]{m})$, where $m$ is a cube-free integer. We can write $m = hk^2$ with $h, k$ coprime integers. The expression for the discriminant of $L$ depends on $m$ modulo 9:
\begin{enumerate}[(i)]
    \item If $m \not \equiv \pm 1$ modulo 9, then $\disc(L) = -27(hk)^2$. Using the same notation as in \cite[\S 5]{BARRUCAND19707}, We will refer to these fields as \textit{Dedekind type I fields.}
    \item If $m \equiv \pm 1$ modulo 9, then $\disc(L) = -3(hk)^2$. As in \cite[\S 5]{BARRUCAND19707}, we will call these fields \textit{Dedekind type II fields}.
\end{enumerate}
Thus, the discriminant of a pure cubic number field is of the form $-3n^2$, $n \in \ZZ$. Conversely, if the discriminant of a cubic number field is of the form $-3n^2$, then it is a pure cubic number field. Indeed, $\disc(L)$ can be written uniquely as $df^2$, where $d$ is a fundamental discriminant, and pure cubic number fields correspond to $d = -3$ (see \cite[\S 6.4.5]{Cohen}). We shall fix $D = -3n^2$ for some $n \in \ZZ$ and a point $P = (x_0, y_0)$ on $E^{-27D}(\QQ)$.

\begin{lemma}
    \label{lemma: cocycle pure field}
    $L_P$ is a pure cubic number field.
\end{lemma}
\begin{proof}
    The discriminant of $f$ is $-y_0^2D$. Write $L_P = \QQ(\alpha)$. Since $\Delta(f) = [\mathcal{O}_{L_P}:\ZZ[\alpha]]^2\disc(L_P)$, we have that $\disc(L_P)$ is of the form $-3n'^2$ for some $n' \in \ZZ$. Then, the result follows from the previous discussion.
\end{proof}
\begin{remark}
    \label{remark: Trivially mono}
    Since $-27D$ is a square, we have $E^{-27D}(\QQ)[\hat{\phi}_D] = \scprod{(0, \sqrt{-27D})}$. The image of this orbit in $S^{\phi_D}(\QQ)$ is, up to sign, the class of  
    $$
    \sigma \mapsto \log_{\mu_3}\lpar{\frac{\sigma(\sqrt[3]{D})}{\sqrt[3]{D}}}(0:\sqrt{D}:1), 
    $$
    i.e. $L_{(0, \sqrt{-27D})} = \QQ(\sqrt[3]{D})$. If $3|n$ and $n/3 \not \equiv \pm 1$ modulo 9, then this cocycle is $\xi_{\QQ(\sqrt[3]{n/3})}$, which is always monogenic since its index form is $X^3 - (n/3)Y^3$. We shall refer to these fields as \textit{trivially monogenic}.
\end{remark}

Now, our goal is to compute a primitive element of $L$ of the form $\sqrt[3]{m}$, with $m \in \ZZ$ square-free. By the substitution $\alpha \mapsto \alpha - x_0/3$, we note that another defining polynomial for $L$ is 
$$
g(X) = X^3 - \frac{x_0^2}{3}X + \frac{27D - 2x_0^3}{27}.
$$
Let $G$ be the homogenization of $g$. By example (\ref{proposition: irreducibility of index form}) and lemma (\ref{lemma: cocycle pure field}), we know that $G = \gamma\star F$, where $F$ is of the form $X^3 - mY^3$ for some cube-free integer $m$. We write $\gamma_{i,j}$ for the coefficient in the $i$-th row and $j$-th column of $\gamma$, $i, j \in \{1,2\}$. 
\begin{lemma}
    We can write $\gamma_{1,2}, \gamma_{2,1}, \gamma_{2,2}$ in terms of $\gamma_{1,1}$, $n$ and $m$ as follows:
    \begin{align*}
        \gamma_{1,2} &= \frac{x_0^2}{9m\gamma_{1,1}}\\
        \gamma_{2,1} &= \frac{(2x_0^3 - 27D \pm 9y_0n)\gamma_{1,1}}{2x_0^2}\\
        \gamma_{2,2} &= \frac{2x_0^3 - 27D \mp 9y_0n}{18m\gamma_{1,1}}
    \end{align*}
\end{lemma}
\begin{proof}
    We have $G = \gamma \star F$, so $h_G(X,Y) = h_{\gamma\star F}(X,Y) = h_F((X,Y)\cdot\gamma)$. Equating the corresponding terms, we obtain the following equalities: 
    \begin{align*}
        X^2:& -x_0^2 = -9m\gamma_{1,1}\gamma_{1,2}\\
        XY:& \frac{-2x_0^3 + 27D}{3} = -9m(\gamma_{1,1}\gamma_{2,2} + \gamma_{1,2}\gamma_{2,1})\\
        Y^2:&\frac{-x_0^4}{9} = -9m\gamma_{2,1}\gamma_{2,2}
    \end{align*}
    From the first equality, we get $\gamma_{1,2} = \frac{x_0^2}{9m\gamma_{1,1}}$. From the second equality, 
    $$
    \gamma_{2,2} = \frac{2x_0^3 - 27D}{27m\gamma_{1,1}} - \frac{\gamma_{1,2}\gamma_{2,1}}{\gamma_{1,1}} = \frac{2x_0^3\gamma_{1,1} - 3x_0^2\gamma_{2,1} - 27\gamma_{1,1}D}{27m\gamma_{1,1}^2}.
    $$
    Substituting $\gamma_{2,2}$ by this expression in the third equality, we get 
    $$
    x_0^4\gamma_{1,1}^2 + 3(27D - 2x_0^3)\gamma_{1,1}\gamma_{2,1} + 9x_0^2\gamma_{2,1}^2 = 0,
    $$
    which is a quadratic equation in $\gamma_{2,1}$. The discriminant $\Delta$ of this equation is $(27y_0\gamma_{1,1}n)^2$, which is in $(\QQ^{\times})^2$. Thus, 
    $$
    \gamma_{2,1} = \frac{(2x_0^3 - 27D \pm 9y_0n)\gamma_{1,1}}{6x_0^2}.
    $$
    Finally, 
    \begin{equation*}
        \gamma_{2,2} = \frac{2x_0^3\gamma_{1,1} - 3x_0^2\gamma_{2,1} - 27\gamma_{1,1}D}{27m\gamma_{1,1}^2} = \frac{2x_0^3 - 27D \mp 9y_0n}{54m\gamma_{1,1}} \qedhere
    \end{equation*}
\end{proof}
\begin{proposition}
    \label{theorem: primitive element}
    Let $D = -3n^2$ for some $n \in \ZZ$. Let $P = (x_0,y_0)$ be a rational point on $E^{-27D}$ and let $L_P$ be the real field with defining polynomial $X^3 - x_0X^2 + D$. Then,
    $$
    L_P = \QQ\lpar{\sqrt[3]{\frac{y_0 - 9n}{y_0 + 9n}}}.
    $$
\end{proposition}
\begin{proof}
We have that the term in $X^2Y$ of $\gamma \star F$ is 
$$
\frac{3\gamma_{1,1}^2\gamma_{2,1}-3m\gamma_{1,2}^2\gamma_{2,2}}{\det(\gamma)}.
$$
Since $\gamma\star F = G$, this coefficient has to be 0. Thus, 
$$
m = \frac{\gamma_{1,1}^2\gamma_{2,1}}{\gamma_{1,2}^2\gamma_{2,2}} = \frac{729(2x_0^3 - 27D \pm 9y_0n)}{x_0^6(2x_0^3 - 27D \mp 9y_0n)}m^3\gamma_{1,1}^6, 
$$
so 
\begin{align*}
    m &= \pm \lpar{\frac{x_0}{3\gamma_{1,1}}}^3\sqrt{\frac{2x_0^3 - 27D \mp 9y_0n}{2x_0^3 - 27D \pm 9y_0n}} = \pm \lpar{\frac{x_0}{3\gamma_{1,1}}}^3\sqrt{\frac{y_0^2 - 81n^2 \mp 18y_0n}{y_0^2 - 81n^2 \pm 18y_0n}}\\
    &= \pm \lpar{\frac{x_0}{3\gamma_{1,1}}}^3\frac{y_0 \mp 9n}{y_0 \pm 9n}.
\end{align*}
Any rational number $m'$ equivalent to $m$ in $\QQ^{\times}/(\QQ^{\times})^3$ satisfies $\QQ(\sqrt[3]{m'}) = \QQ(\sqrt[3]{m})$, so 
$$
\QQ(\alpha) = \QQ\lpar{\sqrt[3]{\frac{y_0 \mp 9n}{y_0 \pm 9n}}} = \QQ\lpar{\sqrt[3]{\frac{y_0 - 9n}{y_0 + 9n}}}
$$
\end{proof} 

\begin{remark}
    From this proposition, we can obtain a primitive element for $L$ of the form $\sqrt[3]{hk^2}$, with $h$ and $k$ coprime and square-free integers. To do this, let $\pm p_1^{e_1}\cdots p_r^{e_r}$ be the prime decomposition of $(y_0 - 9n)/(y_0 + 9n)$ (since this is a rational number, we allow the exponents $e_i$ to be negative). Then, 
    $$
    \pm p_1^{e_1}\cdots p_r^{e_r} \equiv p_1^{e_1\bmod 3}\cdots p_r^{e_r\bmod 3} \quad \textnormal{in }\QQ^{\times}/(\QQ^{\times})^3.
    $$
    The latter is a cube free integer and its cubic root is a primitive element for $L$, as we wanted. From now on, we will denote $p_1^{e_1\bmod 3}\cdots p_r^{e_r\bmod 3}$ by $m$.
\end{remark}
\begin{proposition}
    \label{proposition: factors of n}
    Any prime factor of $m$ different from 3 divides $n$.
\end{proposition}
\begin{proof}
    Write $x_0 = a/b$ and $y_0 = c/d$, $(a,b) = 1$ and $(c,d) = 1$. Then,
    $$
    \frac{y_0 - 9n}{y_0 + 9n} = \frac{c - 9nd}{c + 9nd}.
    $$
    Write $\frac{c - 9nd}{c + 9nd} = p_1^{e_1}\cdots p_r^{e_r}$. Suppose $p_i \nmid n$. We distinguish two cases: 
    \begin{itemize}
        \item If $p_i = 2$, suppose $2$ divides either $c - 9nd$ or $c + 9nd$. If $2 | d$, then $0 \equiv c \pm 9nd \equiv c$ modulo 2, so $c$ and $d$ would have 2 as a common factor. Thus, $2 \nmid c, d$, so 
        $$
        1 = \nu_2(18nd) = \nu_2(c + 9nd - (c - 9nd)) = \min\{\nu_2(c + 9nd), \nu_2(c - 9nd)\}.
        $$
        Since $(x_0, y_0)$ satisfies $y_0^2 = 4x_0^3 - 27D$, we have $(c - 9nd)(c + 9nd)b^3 = 4a^3d^2$. This means
        $$
        1 + \max\{\nu_2(c + 9nd), \nu_2(c - 9nd)\} + 3\nu_2(b) = 3\nu_2(a) + 2,
        $$
        which implies that $e_i = \nu_2\lpar{\frac{y_0 - 9n}{y_0 + 9n}} \equiv 0$ modulo 3. Therefore, $2 \nmid m$.
        \item Otherwise, suppose $p_i | c - 9nd$. Then, $p_i\nmid c + 9nd$, since the contrary would imply $c \equiv 9nd \equiv 0$ modulo $p_i$. This means 
        $$
        \nu_{p_i}(c - 9nd) + 3\nu_{p_i}(b) = 3\nu_{p_i}(a),
        $$
        so $e_i \equiv 0$ modulo 3. The same argument applies when $p_i|c + 9nd$.
    \end{itemize}
    Thus, in any case, if $p_i \nmid n$, then $p_i\nmid m$. This means any $p_i$ dividing $m$ will also divide $n$. 
\end{proof}

\subsection{Dedekind type II fields}
\label{section: Dedekind II}
Suppose $3 \nmid n$ (for example, when $D$ is the discriminant of a Dedekind type II field). Remark (\ref{remark: Trivially mono}) states that $L_{(0, \sqrt{-27D})} = \QQ(\sqrt[3]{D})$, which satisfies $3|D$, so the condition ``different from 3" in the above proposition cannot be improved in general. Apart from $L_{(0, \sqrt{-27D})}$, there are other instances when $3|m$.
\begin{example}
    Let $D = -300 = -3\cdot(10)^2$. A computation in Sage \cite{sagemath} shows that the free part of the Mordell-Weil group of $E^{-27D}/\QQ$ is generated by $P = (-9, 72)$. The primitive element of $L_P$ given by proposition (\ref{theorem: primitive element}) is $-1/9$, so $L_P = \QQ(\sqrt[3]{3})$.
\end{example}

Let $p_0 = 3$ and $p_1,\dots, p_s$ the prime factors of $n$. Proposition (\ref{proposition: factors of n}) implies that we can write
\begin{equation}
    \label{equation: Cocycle expression}
    \xi_P = (\overline{\lambda}_0f_0 + \overline{\lambda}_1f_1 + \cdots + \overline{\lambda}_{s}f_{s})P_D, \quad \textnormal{where }f_i: \sigma \mapsto \log_{\zeta_3}\lpar{\frac{\sigma(\sqrt[3]{p_i})}{\sqrt[3]{p_i}}},
\end{equation}
for some $\overline{\lambda}_i \in \mathbb{F}_3$, $i = 0, \dots, s$. 
\begin{lemma}
    \label{lemma: unicity of coefficients}
    Suppose $\xi = (\sum_{i = 0}^{s}\overline{\lambda}'_if_i)P_D$ is another cocycle which yields the same class in $S^{\phi_D}(E^D/\QQ)$ as $\xi_{P}$. Then, $\overline{\lambda}'_i = \overline{\lambda}_i$ for $i = 0,\dots, s$.
\end{lemma}
\begin{proof}
    Under these hypotheses, $\xi - \xi_P$ is a coboundary. Suppose it is not 0. Then, it is equal, up to sign, to the cocycle 
    $$
    \delta_{P_D} : \sigma \mapsto \lpar{\frac{\sigma(\sqrt{-3})}{\sqrt{-3}} - 1}P_D.
    $$
    Let $D' = p_0^{\lambda'_0 - \lambda_0}\cdots p_s^{\lambda'_s - \lambda_s}$. Since $\xi - \xi_P$ is not zero, $D'$ is different from 1. By construction, $\xi - \xi_P$ is equal to 
    $$
    \sigma \mapsto \log_{\mu_3}\lpar{\frac{\sigma(\sqrt[3]{D'})}{\sqrt[3]{D'}}}P_D.
    $$
    Let $\tau \in G_{\QQ}$ such that $\tau(\sqrt[3]{D'}) = \sqrt[3]{D}$ and $\tau(\sqrt{-3}) = -\sqrt{-3}$. Then, the expression for $\xi - \xi_P$ shows that $(\xi - \xi_P)(\tau) = 0$, whereas $\delta_{P_D}(\tau)\neq 0$. Hence, $\xi - \xi_P$ has to be 0 or, equivalently, $\overline{\lambda}'_i = \overline{\lambda}_i$ for $i = 0,\dots, s$.
\end{proof}

Given generators $P_0 = (0, \sqrt{-27D}), P_1, \dots, P_{r}$ of $E^{-27D}(\QQ)$, where $r$ is the rank of $E^{-27D}(\QQ)$, we construct a matrix $\mathcal{M}$ as follows: we denote $\xi_{P_i} = (\sum_{j = 0}^{s}\overline{\lambda}_j^if_i)(0:\sqrt{D}:1)$. Then, the $i$-th row of $\mathcal{M}$ is $\overline{\lambda}^i := (\overline{\lambda}_0^i, \dots, \overline{\lambda}_{s}^i)$. Thus, $\mathcal{M}$ is an $(r + 1)\times(s + 1)$ matrix over $\mathbb{F}_3$. For $\overline{\lambda} \in \mathbb{F}_3$, let $\lambda$ denote the lift of $\overline{\lambda}$ to $\{0,1,2\} \subset \ZZ$. 
\begin{theorem}
    \label{theorem: Iso Dedekind II}
    $E^{-27D}(\QQ)/\phi_D(E^D(\QQ))$ is isomorphic, as an $\FF_3$-vector space, to the vector space $V$ generated by the rows of $\mathcal{M}$. Moreover, there is a bijection
    \begin{align*}
        \left\{\parbox{18em}{$\scprod{(\overline{\lambda}_0, \dots, \overline{\lambda}_{s})}\subseteq V \ | \ \overline{\lambda}_0 = 0, \overline{\lambda}_i \neq 0$ \textnormal{for} $i = 1, \dots, s \textnormal{ and }p_1^{\lambda_1}\cdots p_{s}^{\lambda_{s}}\equiv\pm 1\mod 9$}\right\}&\longleftrightarrow
        \left\{\parbox{10em}{\textnormal{Quasi-monogenic cubic number fields of\\ discriminant }D}\right\}\\
        \scprod{(\overline{\lambda}_0, \dots, \overline{\lambda}_{s})}&\longmapsto\QQ\lpar{\sqrt[3]{p_1^{\lambda_1}\cdots p_{s}^{\lambda_{s}}}}
    \end{align*}   
\end{theorem}
\begin{proof}
    The map $\delta$ in the exact sequence (\ref{equation: Selmer-TS sequence}) defines an isomorphism of $\mathbb{F}_3$-vector spaces $E^{-27D}(\QQ)/\phi_D(E^D(\QQ))\to \im(\delta)$. We define a map $\im(\delta) \to V$ as follows: to any $\xi_P \in \im(\delta)$, assign the vector $(\overline{\lambda}_0, \dots, \overline{\lambda_s}) \in V$, where $\overline{\lambda}_i$ are the coefficients described in (\ref{equation: Cocycle expression}). This map is well defined by lemma (\ref{lemma: unicity of coefficients}) and it follows from formula (\ref{equation: Cocycle expression}) that it is an isomorphism of $\mathbb{F}_3$-vector spaces.
    
    The second part of the theorem is a direct consequence of the fact that the conditions imposed on the $\overline{\lambda}_i$ imply that the resulting cocycle is associated to a quasi-monogenic Dedekind type II field with discriminant $D$. Conversely, the cocycle associated to such a field satisfies the conditions described above (note that the conditions on $D$ force any cubic number field of discriminant $D$ to be of Dedekind type II).  
\end{proof}

The reduced row echelon form of $\mathcal{M}$ is, up to permutations of the last $s$ columns, 
\begin{equation}
    \label{equation: reduced row echelon form}
    \mathcal{M}_{red} = \left(
    \begin{array}{c|c}
        \id_\rho&*\\\hline
        0&0
    \end{array}
    \right),
\end{equation}
where $\rho = \dim_{\mathbb{F}_3}E^{-27D}(\QQ)/\phi_D(E^D(\QQ))$. By theorem (\ref{theorem: Iso Dedekind II}), if a vector $v$ in $V$ corresponds to a quasi-monogenic Dedekind type II field, then it is of the form $\sum_{i = 0}^s\mu_i\overline{\lambda}'^i$, where $\overline{\lambda}'^i$ is the $i$-th row of $\mathcal{M}_{red}$, $\mu_0 = 0$ and $\mu_i \neq 0$ for $i = 1, \dots, s$. Taking into account that both $v$ and $-v$ determine the same field, we have the following corollary. 
\begin{corollary}
    The number of quasi-monogenic Dedekind type II fields of discriminant $D$ is bounded by $2^{\rho-2}$. \qed
\end{corollary}
\begin{remark}
    Note that $\rho \leq r + 1$, so $2^{r - 1}$ is also a bound for the total number of quasi-monogenic Dedekind type II fields of discriminant $D$.
\end{remark}

\subsection{Dedekind type I fields}
\label{subsection: Main Dedekind type I results}
Now, suppose $3|n$ (for example, when $D$ is the discriminant of a Dedekind type I field). Write $D = -27n'^2$. In this case, we have an improvement of (\ref{proposition: factors of n}).

\begin{proposition}
    \label{proposition: factors of n Dedekind I}
    Any prime factor of $m$ divides $n'$.
\end{proposition}
\begin{proof}
    In view of (\ref{proposition: factors of n}), we just have to prove that if $3|m$, then $3|n'$. Suppose $3\nmid n'$ and that 3 occurs in the factorization of $(y_0 - 9n)/(y_0 + 9n)$. With the same notation as in the proof of (\ref{proposition: factors of n}), we have $(c - 27n'd)(c + 27n'd)b^3 = 4a^3d^2$. Necessarily, $3|c$, so $3\nmid d$. This time, we have $3 = \nu_3(54n'd) = \min\{\nu_3(c + 27n'd), \nu_3(c - 27n'd)\}$, so 
    $$
    3 + \max\{\nu_3(c + 27n'd), \nu_3(c - 27n'd)\} + 3\nu_3(b) = 3\nu_3(a).
    $$
    This implies that $\nu_3((y_0 - 9n)/(y_0 + 9n))\equiv 0$ modulo 3, so 3 does not occur in the factorization of $m$.
\end{proof}
Now, proposition (\ref{proposition: factors of n Dedekind I}) implies that we can write
$$
\xi_P = (\overline{\lambda}_1f_1 + \cdots + \overline{\lambda}_{s}f_{s})P_D.
$$
Given generators $P_1, \dots, P_r$ for the free part of $E^{-27D}(\QQ)$, we construct a matrix $\mathcal{M}$ as follows: denote $\xi_{P_i} = (\sum_{j = 0}^{s}\overline{\lambda}_j^if_j)P_D$. Then, the $i$-th row of $\mathcal{M}$ is $\overline{\lambda}_i := (\overline{\lambda}_1^i, \dots, \overline{\lambda}_{s}^i)$. This time, $\mathcal{M}$ is an $(r + 1)\times s$ matrix over $\mathbb{F}_3$. Applying the same arguments as in theorem (\ref{theorem: Iso Dedekind II}), we obtain the following result. 
\begin{theorem}
    \label{theorem: Iso Dedekind I}
    $E^{-27D}(\QQ)/\phi_D(E^D(\QQ))$ is isomorphic, as an $\FF_3$-vector space, to the vector space $V$ generated by the rows of $\mathcal{M}$. Moreover, there is a bijection
    \begin{align*}
        \pushQED{\qed}
        \left\{\parbox{18em}{$\scprod{(\overline{\lambda}_1, \dots, \overline{\lambda}_{s})}\subseteq V \ | \ \overline{\lambda}_i \neq 0$ \textnormal{for} $i = 1, \dots, s \textnormal{ and }p_1^{\lambda_1}\cdots p_{s}^{\lambda_{s}}\not\equiv\pm 1\mod 9$}\right\}&\longleftrightarrow
        \left\{\parbox{10em}{\textnormal{Quasi-monogenic cubic number fields of\\ discriminant }D}\right\}\\
        \scprod{(\overline{\lambda}_0, \dots, \overline{\lambda}_{s})}&\longmapsto\QQ\lpar{\sqrt[3]{p_1^{\lambda_1}\cdots p_{s}^{\lambda_{s}}}}\qedhere
        \popQED
    \end{align*}
\end{theorem}
Now, consider the reduced row echelon form of $\mathcal{M}$ as in (\ref{equation: reduced row echelon form}). If a vector $v$ is associated to a quasi-monogenic Dedekind type I field with discriminant $D$, then it is of the form $\sum_{j = 0}^s\mu_j\overline{\lambda}'^i$, where $\overline{\lambda}'^i$ is the $i$-th row of $\mathcal{M}_{red}$ and $\mu_j \neq 0$ for $j = 1,\dots, s$. Thus, we get the following corollary.
\begin{corollary}
    The number of quasi-monogenic Dedekind type I fields of discriminant $D$ is bounded by $2^{\rho - 1}$. \qed
\end{corollary}
\begin{remark}
    As before, since $\rho \leq r + 1$, we have that the number of quasi-monogenic Dedekind type I fields of discriminant $D$ is bounded by $2^r$. There are examples where this bound is reached, as shown in the next example.  
\end{remark}
\begin{example}
    Let $D = -27n^2$, where $n = 2\cdot 3\cdot 5 = 30$. A computation on Sage gives the following set of generators for the free part of $E^{-27D}(\QQ)$,
    $$
    \{(-54 : 81 : 1), (-45 : 270 : 1)\}.
    $$
    This yields the following generator matrix: 
    $$
    \mathcal{M} = 
    \begin{pmatrix}
        1&1&1\\
        1&2&0\\
        0&0&1
    \end{pmatrix}.
    $$
    Its reduced row echelon form is the identity matrix, so there are 4 quasi-monogenic cubic number fields with discriminant $D$. These are $\QQ(\sqrt[3]{n})$, $\QQ(\sqrt[3]{60})$, $\QQ(\sqrt[3]{90})$ and $\QQ(\sqrt[3]{150})$. Solving the index form equations, we have that all of these are monogenic.
\end{example}

\section{Computing monogenity of families of cubic number fields with large discriminant (assuming GRH)}
\label{section: computations}

We shall apply the techniques developed to study the monogenity of numerically difficult cubic number fields. In the LMFDB repository, there are 4619 cubic number fields whose monogenity is still unknown as of today (see \cite[\href{https://www.lmfdb.org/NumberField/?degree=3&monogenic=unknown}{List of cubic number fields with unknown monogenity}]{lmfdb}). Of these, 4569 are pure. We shall focus on these fields. Firstly, we will focus on Dedekind type I fields and, secondly, on Dedekind type II fields.

As hinted in the previous sections, the costliest part of our computations will be to determine the Mordell-Weil group of the elliptic curves $E^{-27D}(\QQ)$. We shall assume GRH to do so in a reasonable amount of time. We propose two options to compute $E^{-27D}(\QQ)$ using Magma \cite{magma}. The first one is to compute it directly, for example, from this code in Magma,

\begin{verbatim}
> SetClassGroupBounds("GRH");
> D := -1086061775432017340256300;
> E := EllipticCurve([0,-27*D/4]);
> Generators(E)
\end{verbatim}

The runtime for this computation in some large examples is the following.

\begin{table}[h]
\begin{tabular}{|l|l|}
\hline
\multicolumn{1}{|c|}{$D$} & \multicolumn{1}{c|}{Time (s)} \\ \hline
-14572337954782934700     & 2765.58                       \\ \hline
-20353100118663768300     & 2615.74                       \\ \hline
-32787760398261603075     & 2632.40                       \\ \hline
-61397579028323666700     & 2687.52                       \\ \hline
\end{tabular}
\caption{Timings for computing the Mordell-Weil Group of $E^D$. Times are in seconds for a single thread running on a 3.6 GHz Intel Xeon CPU.}
\end{table}

However, it is substantially faster to compute $E^D(\QQ)$ instead. In our cases, the maximum running time for computing the Mordell-Weil group of $E^D(\QQ)$ is less than 7 seconds. This leads to our second option, which is guaranteed to work if $\hat{\phi}_D(E^{-27D}(\QQ)) = E^D(\QQ)$. Given this condition, take $r$ generators for the free part of $E^D(\QQ)$, where $r$ is the rank of $E^D$. Then, take their preimages by $\hat{\phi}_D$. These cannot have a preimage by $\phi_D$, so we obtain $r$ generators of $E^{-27D}(\QQ)$. 

\subsection{Dedekind type I fields}

Out of the 4619 fields mentioned above, 4389 are of Dedekind type I. The discriminants of each of these are all included in table (\ref{table: Dedekind type I}). In all of these cases, we have $\hat{\phi}_D(E^{-27D}(\QQ)) = E^D(\QQ)$. In order to compute the quasi-monogenic Dedekind type I fields with each discriminant, we shall follow the following steps: 
\begin{enumerate}
    \item Compute generators of $E^{D}(\QQ)$ (assuming $GRH$ to speed up computations) and compute their preimages by $\hat{\phi}_D$.
    \item From these generators, compute a generator matrix $\mathcal{M}$ as described in (\ref{subsection: Main Dedekind type I results}).
    \item Recover the quasi-monogenic cubic number fields using theorem (\ref{theorem: Iso Dedekind I}).
\end{enumerate}
The results are showcased in table (\ref{table: Dedekind type I}).
\newpage

\begin{table}[H]
\begin{tabular}{|l|c|c|l|}
\hline
\multicolumn{1}{|c|}{discriminant} & \multicolumn{1}{c|}{\begin{tabular}[c]{@{}c@{}}rank\\ (GRH)\end{tabular}} & \multicolumn{1}{c|}{\begin{tabular}[c]{@{}c@{}}number of \\ fields\end{tabular}} & \multicolumn{1}{c|}{\begin{tabular}[c]{@{}c@{}}Quasi-monogenic\\ cubic number fields\\ $(\QQ(\sqrt[3]{\cdot}))$\end{tabular}} \\
\hline
-14572337954782934700 & 2 & 84 &-\\
\hline
-20353100118663768300 & 2 & 84 & 868227230 (*)\\
\hline
-32787760398261603075 & 1 & 128 & 1101980715 (*)\\
\hline
-51655732481903321772 & 0 & 256 & 1383175794 (*)\\
\hline
-61397579028323666700 & 2 & 256 & 1507973610 (*)\\
\hline
-76693861692819528300 & 3 & 256 & 1685382270 (*)\\
&&& 1358418109620 \\
\hline
-83568927010773879675 & 2 & 86 & 1759302545 (*)\\
&&& 1257901319675 \\
\hline
-120339254895514386732 & 1 & 256 & 2111163054 (*)\\
\hline
-183177901067973914700 & 3 & 256 & 2604681690 (*)\\
&&& 39807350268270 \\
&&& 58050540825030 \\
\hline
-322848328011895761075 & 2 & 256 & 3457939485 (*)\\
\hline
-752120343096964917075 & 3 & 256 & 5277907635 (*)\\
&&& 131710185031425 \\
&&& 156579685807545 \\
\hline
-1130137123238311488300 & 2 & 512 & 6469693230 (*)\\
\hline
-4826941224142299290028 & 4 & 172 & 13370699342 (*)\\
&&& 171385624165756 \\
&&& 228384915460702 \\
\hline
-5541131507306210919675 & 1 & 256 & 14325749295 (*)\\
\hline
-6426401038059274202700 & 4 & 512 & 15427730010 (*)\\
&&& \textbf{241999372936860} \\
&&& \textbf{308955721180260} \\
&&& \textbf{977578112083650} \\
&&& \textbf{1844647394105670} \\
\hline
-8975717152330721820300 & 1 & 512 & 18232771830 (*)\\
\hline
-10860617754320173402563 & 0 & 256 & 20056049013 (*)\\
\hline
-22164526029224843678700 & 2 & 512 & 28651498590 (*)\\
\hline
-30168382650889370562675 & 2 & 172 & 33426748355 (*)\\
\hline
-43442471017280693610252 & 3 & 512 & 40112098026 (*)\\
&&& \textbf{77737245974388} \\
\hline
-120673530603557482250700 & 2 & 340 &-\\
\hline
-271515443858004335064075 & 3 & 512 & 100280245065 (*)\\
&&& \textbf{71700375221475} \\
\hline
-1086061775432017340256300 & 2 & 1024 & 200560490130 (*)\\
\hline
\end{tabular}
\caption{A summary of all the families of Dedekind type I fields with the same discriminant such that at least one field has unknown monogenity in LMFDB. Each row contains a discriminant $D$, the rank of $E^D$ assuming GRH, the number of fields with discriminant $D$ up to isomorphism and the corresponding quasi-monogenic cubic number fields. The trivially monogenic ones are marked with an asterisk (*). The ones whose monogenity remains unknown are written in bold.}
\label{table: Dedekind type I}
\end{table}

The total number of fields analyzed is 7466. Out of these, 35 are quasi-monogenic, listed in the table. 22 of them are trivially monogenic, so this leaves us with 13 fields whose monogenity cannot be determined using our techniques. Upon checking the LMFDB repository, we note that only 6 of these fields have unknown monogenity, which are marked in the table. 

\subsection{Dedekind type II fields}

Out of the 4619 cubic number fields with unknown monogenity, 175 are of Dedekind type II. The discriminants of these fields are included in table (\ref{table: Dedekind type II}). 172 of the 175 fields just mentioned have discriminant $-13408170067061942472300$. We proceed as in the Dedekind type I case, but now we apply theorem (\ref{theorem: Iso Dedekind II}) instead. The results can be seen in table (\ref{table: Dedekind type II}). In total, we have analysed 341 fields, none of which are quasi-monogenic.

\begin{table}[H]
\begin{tabular}{|l|c|c|l|}
\hline
\multicolumn{1}{|c|}{discriminant} & \multicolumn{1}{c|}{\begin{tabular}[c]{@{}c@{}}rank\\ (GRH)\end{tabular}} & \multicolumn{1}{c|}{\begin{tabular}[c]{@{}c@{}}number of \\ fields\end{tabular}} & \multicolumn{1}{c|}{\begin{tabular}[c]{@{}c@{}}Quasi-monogenic\\ cubic number fields\\ $(\QQ(\sqrt[3]{\cdot}))$\end{tabular}} \\
\hline
-12627147759764869116703083 & 0 & 1 &-\\
\hline
-13408170067061942472300 & 3 & 172 &-\\
\hline
-15943127309229420300 & 2 & 84 &-\\
\hline
-13952310163435944300 & 1 & 84 &-\\
\hline
\end{tabular}
\caption{A summary of all the families of Dedekind type I fields with the same discriminant such that at least one field has unknown monogenity in LMFDB. Note that none of the fields are quasi-monogenic.}
\label{table: Dedekind type II}
\end{table}

\subsection*{Acknowledgements} The second  author gratefully acknowledges the Universitat Politècnica de Catalunya and Banco Santander for the financial support of his predoctoral FPI-UPC grant. The first author is supported by projects PID2022-136944NB-I00 and  2021 SGR 01468.

\newpage
\bibliographystyle{plain}
\bibliography{Determining_Monogenity_of_Pure_Cubic_Fields_V2}

\begin{thebibliography}{10}

\bibitem{Bhargava}
Levent Alp\"oge, Manjul Bhargava, and Ari Shnidman.
\newblock A positive proportion of cubic fields are not monogenic yet have no
  local obstruction to being so.
\newblock {\em Math. Ann.}, 391(4):5535--5551, 2025.

\bibitem{ArpinI}
Sarah Arpin, Sebastian Bozlee, Leo Herr, and Hanson Smith.
\newblock The scheme of monogenic generators {I}: representability.
\newblock {\em Res. Number Theory}, 9(1):Paper No. 14, 33, 2023.

\bibitem{BARRUCAND19707}
Pierre Barrucand and Harvey Cohn.
\newblock A rational genus, class number divisibility, and unit theory for pure
  cubic fields.
\newblock {\em Journal of Number Theory}, 2(1):7--21, 1970.

\bibitem{Bhargava_2019}
Manjul Bhargava, Noam Elkies, and Ari Shnidman.
\newblock The average size of the 3‐isogeny selmer groups of elliptic curves
  y2=x3+k.
\newblock {\em Journal of the London Mathematical Society}, 101(1):299–327,
  August 2019.

\bibitem{magma}
Wieb Bosma, John Cannon, and Catherine Playoust.
\newblock The {M}agma algebra system. {I}. {T}he user language.
\newblock {\em J. Symbolic Comput.}, 24(3-4):235--265, 1997.
\newblock Computational algebra and number theory (London, 1993).

\bibitem{Cohen}
Henri. Cohen.
\newblock {\em A Course in computational algebraic number theory}.
\newblock Graduate texts in mathematics ; 138. Springer-Verlag, Berlin, 1993.

\bibitem{CREUTZ2012673}
Brendan Creutz and Robert~L. Miller.
\newblock Second isogeny descents and the birch and swinnerton-dyer conjectural
  formula.
\newblock {\em Journal of Algebra}, 372:673--701, 2012.

\bibitem{Delone2009TheTO}
B.~N. Delone, D.~K. Faddeev, Emma Lehmer, and Susan~A. Walker.
\newblock The theory of irrationalities of the third degree.
\newblock 2009.

\bibitem{AAlgebra}
David~Steven. Dummit.
\newblock {\em Abstract algebra}.
\newblock John Wiley \& Sons, Hoboken, third edition, 2004.

\bibitem{GanGrossSavin}
Wee~Teck Gan, Benedict Gross, and Gordan Savin.
\newblock {Fourier coefficients of modular forms on G2}.
\newblock {\em Duke Mathematical Journal}, 115(1):105 -- 169, 2002.

\bibitem{Gaal89}
I.~Gaál and N.~Schulte.
\newblock Computing all power integral bases of cubic fields.
\newblock {\em Mathematics of Computation}, 53(188):689--696, 1989.

\bibitem{GAAL1993563}
István Gaál, Attila Pethő, and Michael Pohst.
\newblock On the resolution of index form equations in quartic number fields.
\newblock {\em Journal of Symbolic Computation}, 16(6):563--584, 1993.

\bibitem{GAAL199690}
István Gaál, Attila Pethő, and Michael Pohst.
\newblock Simultaneous representation of integers by a pair of ternary
  quadratic forms—with an application to index form equations in quartic
  number fields.
\newblock {\em Journal of Number Theory}, 57(1):90--104, 1996.

\bibitem{MontesAlgorithm}
Jordi Gu\`ardia, Jes\'us Montes, and Enric Nart.
\newblock Newton polygons of higher order in algebraic number theory.
\newblock {\em Trans. Amer. Math. Soc.}, 364(1):361--416, 2012.

\bibitem{knapp}
Anthony~W. Knapp.
\newblock {\em Elliptic curves}.
\newblock Mathematical notes ; 40. Princeton University Press, Princeton, 1992.

\bibitem{lmfdb}
The {LMFDB Collaboration}.
\newblock The {L}-functions and modular forms database.
\newblock \url{http://www.lmfdb.org}, 2023.
\newblock [Online].

\bibitem{milneANT}
James~S. Milne.
\newblock Algebraic number theory (v3.08), 2020.
\newblock Available at www.jmilne.org/math/.

\bibitem{pdescent}
Edward~F. Shaefer and Michael Stoll.
\newblock How to do a $p$-descent on an elliptic curve.
\newblock {\em preprint}, 2001.

\bibitem{Silverman}
Joseph~H Silverman.
\newblock {\em The Arithmetic of Elliptic Curves}.
\newblock Springer-Verlag, New York, NY, second edition, 2009.

\bibitem{sagemath}
{The Sage Developers}.
\newblock {\em {S}ageMath, the {S}age {M}athematics {S}oftware {S}ystem
  ({V}ersion 10.0)}, 2023.
\newblock {\tt https://www.sagemath.org}.

\bibitem{WILDANGER2000188}
K~Wildanger.
\newblock Über das lösen von einheiten- und indexformgleichungen in
  algebraischen zahlkörpern.
\newblock {\em Journal of Number Theory}, 82(2):188--224, 2000.

\end{thebibliography}

\end{document}